\documentclass{amsart}
\usepackage{amscd}{
\usepackage{amsfonts}
\usepackage{amssymb}
\usepackage{amsthm}
\usepackage{eucal}
\theoremstyle{plain}
\newtheorem{theorem}{Theorem}[section]
\newtheorem{lemma}[theorem]{Lemma}
\newtheorem{corollary}[theorem]{Corollary}
\newtheorem{definition}[theorem]{Definition}
\newtheorem{proposition}[theorem]{Proposition}
\newtheorem{example}[theorem]{Example}

\title{On the multiple exterior degree of finite groups}

\author[R. Rezaei]{Rashid Rezaei}
\address{Department of Mathematics, Faculty of sciences, Malayer university, Malayer, Iran}
\email{ras$\_$rezaei@yahoo.com}
\author[P. Niroomand]{Peyman Niroomand}
\address{School of Mathematics and Computer Science\\
Damghan University, Damghan, Iran}
\email{p$\_$niroomand@yahoo.com, niroomand@du.ac.ir}
\author[A. Erfanian]{Ahmad Erfanian}
\address{Department of Mathematics and Centre of Excellence in Analysis on Algebraic Structures, Ferdowsi University of Mashhad, Mashhad, Iran.}
\email{erfanian@math.um.ac.ir}
\keywords{Exterior degree,  Multiple exterior degree, multiple commutativity degree, Capable groups, Schur multiplier}

\subjclass[2000]{20F99; 20P05.}

\date{}
\begin{document}

\begin{abstract}
Recently, two first authors have introduced a group invariant, which is related
to the number of elements $x$ and $y$ of a finite group $G$ such that $x\wedge y=1$ in the exterior square $G\wedge G$ of $G$.  Research on this probability gives some relations between the concept and Schur multiplier and the capability of finite groups. In the present paper, we will generalize the concept of exterior degree of groups and we will introduce
 the multiple exterior degree of finite groups. Among the other results, we will state some results between the multiple exterior degree, multiple commutativity degree and capability of finite groups.

\end{abstract}

\maketitle
\section{Introduction }
Let $(g,h)\mapsto~ ^hg$ be the conjugation action of $G$ on itself and
$[g,h] = ghg^{-1}h^{-1}$. In \cite{brow1,brow2,brow3} the tensor square $G\otimes G$ is defined as the group
generated by the symbols $g\otimes h$ subject to the relations
\[gk\otimes h=(~^gk\otimes ~^gh)(g\otimes h)~~\text{and}~~~
g\otimes hk=(g\otimes h)(~^hg\otimes ~^hk)~~\text{for all}~~g, k, h\in G.\] The exterior square $G\wedge G$ is obtained by imposing the additional relations $g\otimes g=1$ ($g\in G$) on $G\otimes G$. The definition emphasizes there is an epimorphism $\kappa$ from $G\otimes G$ to $G^{'}$ defined on the generators by $\kappa(g\otimes h)=[g,h]$ for all $g,h\in G$, and its kernel is denoted by $J_{2}(G)$. Let $\nabla(G)$ be the subgroup of $J_{2}(G)$, generated by all
elements $g\otimes g$, for all $g$ in $G$. Then $G\wedge G=G\otimes G/\nabla(G)$ and the image of $x\otimes y$
is denoted by $x\wedge y$  for all $x$ and $y$ in $G$. Since $\kappa(\nabla(G))=1$, $\kappa$
induces the epimorphism $\kappa^{'}$ from $G\wedge G$ onto $G^{'}$
with
 $M(G)$, the Schur multiplier of the group $G$, as its kernel (see Miller \cite{millc}).

In this context, if $x\in G$, the exterior centralizer of $x$ in $G$ is the set $C^{^{\wedge}}_G(x)=\{y\in G ~|~ x\wedge y=1\},$
which turns out to be a subgroup of $G$ and the exterior center of $G$, which is denoted by $Z^{^{\wedge}}(G)$, is the intersection of all exterior centralizers of elements of $G$;(see Bacon and Kappe \cite{baka}). Ellis in \cite{ellis} showed that a group
$G$ is
capable if and only if $Z^{^{\wedge}}(G)=1$.

 The exterior degree of finite groups is introduced in \cite{pr}, by two first authors, which is the probability for two elements of $G$ like $x$ and $y$ such that $x\wedge y=1$. The aim of paper is a generalization of the concept of exterior degree of groups. We define the multiple exterior degree of a finite group and try to find some upper bounds for this degree and
state some relations between multiple exterior degree, multiple commutativity degree and capability  of
groups.
\section{Some known results}

In this section we are reminded some known results for the multiple commutativity degree and exterior degree of finite groups. For a finite group $G$ the multiple commutativity degree of $G$ (denoted by $d_n(G)$) is defined by
 $$d_n(G)=\frac{|\{(x_1,x_2,...,x_{n+1})\in G^{n+1}: x_ix_j=x_jx_i, ~ \textrm{for all} ~ 1\leq i,j \leq n+1     \}|  }{|G|^{n+1}}.$$
 Clearly $d_0(G)=1$ and $G$ is abelian if and only if $d_n(G)=1$. For the case for which $n=1$, $d_1(G)=d(G)$ is called the commutativity degree of group $G$ that is equal to $k(G)/|G|$, where $k(G)$ is the number of conjugacy classes of G. (see Lescot \cite{l2}).

 \begin{lemma}\label{prim0}
 Let $G$ be a finite group. Then for all $n\in\mathbb{N}$,
 $$ d_{n}(G)=\frac{1}{|G|^{n+1}}\sum_{x\in G}|C_G(x)|^{n}d_{n-1}(C_G(x)).$$
 \end{lemma}
 \begin{proof}
 \begin{align*}
d_{n}(G)&= \frac{1}{|G|^{n+1}} |\{(x_1,x_2,\cdots,x_{n+1})\in G^{n+1}:~ x_ix_j=x_jx_i, ~ \textrm{for all} ~ 1\leq i,j \leq n+1\}|\\
&=\frac{1}{|G|^{n+1}} \sum_{x\in G} |\{(x_1,x_2,\cdots,x_{n})\in (C_{G}(x))^{n}: x_ix_j=x_jx_i,  \textrm{for all} ~ 1\leq i,j \leq n \}| \\&=\frac{1}{|G|^{n+1}} \sum_{x\in G}|C_{G}(x)|^n { d_{n-1}(C_{G}(x))}.
\end{align*}
 \end{proof}

 The following Lemma is a consequence of Lemma \ref{prim0}.

 \begin{lemma}\label{prim1}\cite[Lemma 4.1]{l2}
 Let ${g_1,\cdots,g_{k(G)}}$ be a system of representative for the conjugacy classes of $G$. Then for all $n\in\mathbb{N}$,
 $$d_{n+1}(G)=\frac{1}{|G|}\sum_{i=1}^{k(G)}\frac{1}{|cl(g_i)|^n}d_n(C_G(g_i))$$
 \end{lemma}

 \begin{theorem}
Let $N$ be a normal subgroup of finite group $G$. Then, $$d_n(G)\leq  d_n(\frac GN).$$
\end{theorem}
\begin{proof}
Use induction on $n$. In the case for which $n=1$, the result follows by Lescot (\cite{l2} Lemma 1.4). Now by Lemma \ref{prim0} and using inductive hypothesis,
\begin{align*}
{|G|^{n+1}}d_n(G)&=\sum_{x\in G}|C_G(x)|^{n}d_{n-1}(C_G(x))\\
&=\sum_{x\in G}\frac{|C_G(x)N|^{n}}{[C_G(x)N:C_G(x)]^n}d_{n-1}(C_G(x))\\
&=\sum_{S\in G/N}\sum_{x\in S}\left|\frac{C_G(x)N}{N}\right|^{n}|C_N(x)|^nd_{n-1}(C_G(x))\\
&\leq\sum_{S\in G/N}\sum_{x\in S}|C_{G/N}(xN)|^nd_{n-1}(C_{G/N}(xN))|C_N(x)|^n\\
&=\sum_{S\in G/N}|C_{G/N}(S)|^nd_{n-1}(C_{G/N}(xN))\sum_{x\in S}|C_N(x)|^n\\
&\leq|N|^{n+1}\sum_{S\in G/N}|C_{G/N}(S)|^nd_{n-1}(C_{G/N}(S))\\
&=|N|^{n+1}|G/N|^{n+1}d_n(\frac GN).
\end{align*}
Hence $d_n(G)\leq d_n(G/N).$ Now if $N\cap G^{'}=1$, then $d(G)=d(G/N)$, $C_G(x)N/N=C_{G/N}(xN)$ and $d_{n-1}(C_G(x))=d_{n-1}(C_{G/N}(xN))$. Therefore all inequalities should be changed into equalities, and so $d_n(G)=d_n(G/N).$
\end{proof}

\begin{proposition}\label{prim2}
 Let $G$ be a non-abelian group. Then for all $n\in\mathbb{N}$,
$$d_{n}(G)\leq\frac{p^{n+1}+p^n-1}{p^{2n+1}}, $$
and the equality holds if and only if $G/Z(G)$ is an elementary abelian $p$-group of rank $2$.
\end{proposition}

\begin{proof}
We proceed by induction on $n$. For $n=1$ the proof is clear by \cite[Lemma 1.3]{l2}. Now assume that the result is true for $n$. Then by using Lemma \ref{prim0},
\begin{align*}
d_{n}(G)&=\frac{1}{|G|^{n+1}}\sum_{x\in G}|C_G(x)|^{n}d_{n-1}(C_G(x))\\
&=\frac{1}{|G|^{n+1}}\left(\sum_{x\in Z(G)}|C_G(x)|^{n}d_{n-1}(C_G(x))+\sum_{x\in G\setminus Z(G)}|C_G(x)|^{n}d_{n-1}(C_G(x))\right)\\
&\leq\frac{1}{|G|^{n+1}}\left(|Z(G)||G|^nd_{n-1}(G)+\frac{(|G|-|Z(G)|)|G|^n}{p^n}\right)\\
&=\frac{1}{p^n}+\frac{|Z(G)|}{|G|}\left(d_{n-1}(G)-\frac{1}{p^n}\right)\leq\frac{1}{p^n}+\frac{1}{p^2}\left(\frac{p^n+p^{n-1}-1}{p^{2n-1}}-\frac{1}{p^n}
\right)\\
&=\frac{p^{n+1}+p^n-1}{p^{2n+1}}
\end{align*}
and the inequality is proved at most for $n+1$. Assume that $G/Z(G)$ is an elementary abelian $p$-group of rank $2$. We have $d_n(C_G(x))=1$ for all $x\in G\setminus Z(G)$ and by the similar computations as above we can prove that $$d_{n}(G)=\frac{p^{n+1}+p^n-1}{p^{2n+1}}.$$
Conversely, let the equality holds, then
\begin{align*}
\frac{p^{n+1}+p^n-1}{p^{2n+1}}=d_{n}(G)&=\frac{1}{|G|^{n+1}}\sum_{x\in G}|C_G(x)|^{n}d_{n-1}(C_G(x))\\
&\leq\left(\frac{|Z(G)|}{|G|}d_{n-1}(G)+\frac{1}{p^n}(1-\frac{|G|}{|Z(G)|})\right)\\
&=\frac{1}{p^n}+\frac{|Z(G)|}{|G|}\left(d_{n-1}(G)-\frac{1}{p^n}\right)\\
&\leq\frac{1}{p^n}+\frac{|Z(G)|}{|G|}\left(\frac{p^n+p^{n-1}-1}{p^{2n-1}}-\frac{1}{p^n}\right).
\end{align*}
It follows that $|G|/|Z(G)|\leq p^2$ and so by the assumption $G/Z(G)\cong\mathbb{Z}_p\times\mathbb{Z}_p$.
\end{proof}

Note that in the above theorem, if $G$ is a finite non-abelian group of even order, then $d_n(G)$ is at most $(3\cdot2^n-1)/2^{2n+1}$ (see \cite{l2}).

Our result gives a sharper upper bound for $d_n(G)$ under some conditions.

\begin{theorem}\label{theo}
Let $p$ be the smallest prime number dividing the order of finite non-abelian group $G$. If $Z(G)\cap G^{'}=1$, then $$d_n(G)\leq \frac{1}{p^n}.$$
\end{theorem}
\begin{proof}
We proceed by induction on $n$.  Thanks to Proposition 3.3 in \cite{mogh} we have $d(G)\leq 1/p$ and so by Lemma \ref{prim1},
\begin{align*}
d_n(G)&=\sum_{i=1}^{k(G)}\frac{1}{|cl(g_i)|^{n-1}}d_{n-1}(C_G(g_i))\\
&=\frac{1}{|G|}\left(|Z(G)|d_{n-1}(G)+\frac{1}{p^{n-1}}(k(G)-|Z(G)|)\right)\\
&=\frac{1}{p^{n-1}}d(G)+\frac{|Z(G)|}{|G|}(d_{n-1}(G)-\frac{1}{p^{n-1}})=\frac{1}{p^n}.
\end{align*}
\end{proof}

The concept of exterior degree of finite group, $d^{^{\wedge}}(G)$, is defined in  \cite{pr} as the probability for two elements $x$ and $y$ in $G$ such that $x\wedge y=1$. In other words $d^{^{\wedge}}(G)=|C_2|/|G|^2$, where $C_2=\{(x,y)\in G\times G : x\wedge y=1\}$. Furthermore, $G$ is called unidegree, if $d(G)=d^{^{\wedge}}(G)$ and one can show that every unicentral group, $Z(G)=Z^{^{\wedge}}(G)$, is unidegree. (See \cite{pr}).

\begin{lemma}\cite[Lemma 2.2]{pr}
 Let ${g_1,\cdots,g_{k(G)}}$ be a system of representative for the conjugacy classes of $G$. Then
 $$d^{^{\wedge}}(G)=\frac{1}{|G|}\sum_{i=1}^{k(G)}\frac{|C^{^{\wedge}}_G(g_i)|}{|C_G(g_i)|}.$$
 \end{lemma}

 The following upper bounds that have introduced in \cite{pr} will be generalized in the next section.

 \begin{theorem}\label{prim3} \cite[Theorem 2.3]{pr}
 For every finite group $G$, $$d^{^{\wedge}}(G)\leq d(G)-\left(\frac{p-1}p\right)\left(\frac{|Z(G)|-|Z^{^\wedge}(G)|}{|G|}\right),$$
 where $p$ is the smallest prime number dividing the order of $G$.
\end{theorem}

\begin{corollary}\label{coro}\cite[Corollary 2.4]{pr}
 Let $p$ be the smallest prime number dividing the order of $G$, then:\\
(i) if $G$ is non-cyclic and abelian, or $G$ is non-abelian, then $d^{^\wedge}(G)\leq(p^2+p-1)/p^3;$\\
(ii) if $G$ is non-abelian and $Z^{^\wedge}(G)$ is a proper subgroup of $Z(G)$, then $d^{^\wedge}(G)\leq(p^3+p-1)/p^4;$\\
(iii) $G$ is cyclic if and only if $d^{^\wedge}(G)=1. $
\end{corollary}


\section{multiple exterior degree }

In this section we will introduce the multiple exterior degree of a finite group and we will state some results for the new concept.

\begin{definition}
Let $G$ be a finite group, the multiple exterior degree of $G$ is defined as the ratio
$$D^{^{\wedge}}_n(G)= \frac{|\{(x_1,x_2,...,x_{n+1})\in G^{n+1}: x_i \wedge x_j=1,~~  1\leq i,j \leq n+1     \}|  }{|G|^{n+1}}.$$ It is clear that $D^{^{\wedge}}_0(G)=1$ and in the case that $n=1$, we denote $D^{^{\wedge}}_1(G)$ by $d^{^{\wedge}}(G)$, the exterior degree of $G$. Furthermore, one can check that $G$ is cyclic if and only if $D^{^{\wedge}}_n(G)=1$.
\end{definition}
\begin{lemma}\label{main0}
Let $G$ be a finite group. Then for all $n\in \mathbb{N} $,
$$D^{^{\wedge}}_n(G)= \frac{1}{|G|^{n+1}} \sum_{x\in G}|C^{{^\wedge}}_{G}(x)|^n D^{^{\wedge}}_{n-1}(C^{^{\wedge}}_{G}(x)). $$
\end{lemma}
\begin{proof}
\begin{align*}
D^{^{\wedge}}_n(G)&= \frac{1}{|G|^{n+1}} |\{(x_1,x_2,...,x_{n+1})\in G^{n+1}:x_i\wedge x_j=1 ~ \textrm{for all} ~ 1\leq i,j \leq n+1\}|\\
&=\frac{1}{|G|^{n+1}} \sum_{x\in G} |\{(x_1,x_2,...,x_{n})\in (C^{^{\wedge}}_{G}(x))^n:  x_i\wedge x_j=1~  \textrm{for all} ~ 1\leq i,j \leq n\}| \\
&=\frac{1}{|G|^{n+1}} \sum_{x\in G}|C^{^{\wedge}}_{G}(x)|^n { D^{^{\wedge}}_{n-1}(C^{^{\wedge}}_{G}(x))}.
\end{align*}
\end{proof}
\begin{proposition}
$\{D^{^{\wedge}}_n(G) \}_{n\geq1}$ is a descending sequence.
\end{proposition}
\begin{proof}
We may proceed by induction on $n$.
$$D^{^{\wedge}}_2(G)= \frac{1}{|G|} \sum_{x\in G} (\frac{1}{[G:C^{^{\wedge}}_{G}(x) ]})^2 D^{^{\wedge}}_1(C^{^{\wedge}}_{G}(x)) \leq \frac{1}{|G|} \sum_{x\in G} (\frac{1}{[G:C^{^{\wedge}}_{G}(x) ]})= D^{^{\wedge}}_1(G).$$
Now assume that the result holds for $n$, then
\begin{align*}
 D^{^{\wedge}}_{n+1}(G)&= \frac{1}{|G|} \sum_{x\in G} (\frac{1}{[G:C^{^{\wedge}}_{G}(x) ]})^{n+1} D^{^{\wedge}}_n(C^{^{\wedge}}_{G}(x))\\ 
 &\leq \frac{1}{|G|} \sum_{x\in G} (\frac{1}{[G:C^{^{\wedge}}_{G}(x) ]})^nD^{^{\wedge}}_{n-1}(C^{^{\wedge}}_{G}(x))= D^{^{\wedge}}_n(G).
 \end{align*}
\end{proof}
\begin{proposition}\label{prop1}
For all $n\geq 1$, $D^{^{\wedge}}_n(G)\leq d_n(G)$.
\end{proposition}
\begin{proof}
The proof is clear by Lemma \ref{main0} and using induction on $n$.
\end{proof}

In Proposition \ref{prop1} if the equality holds, then $G$ is called multiple unidegree group. By Theorem \ref{main0} and using induction, one can show that if $G$ is unidegree group, $d(G)=d^{^{\wedge}}(G)$, then $G$ is a multiple unidegree group. Furthermore the mapping $f:C_G(x)\rightarrow M(G)$ by the rule $f(y)=x\wedge y$ is a homomorphism with kernel $C^{^{\wedge}}_{G}(x)$. Therefore if $M(G)=1$, then $C^{^{\wedge}}_{G}(x)=C_G(x)$ and so $G$ is multiple unidegree.
\begin{lemma}\label{main1}
Let $\{x_1,x_2,...,x_{k(G)} \}$ be a system of representative for the conjugacy classes of a finite group $G$, then
$$D^{^{\wedge}}_n(G)= \frac{1}{|G|^n} \sum_{i=1}^{k(G)}\frac{|C^{^{\wedge}}_{G}(x_i)|^n}{|C_{G} (x_i) |} ~D^{^{\wedge}}_{n-1}(C^{^{\wedge}}_{G}(x_i)).$$
\end{lemma}
\begin{proof}
Since for every element $g$ in $G$, $C^{^{\wedge}}_{G}(x^g)  = C^{^{\wedge}}_{G}(x)^g  $, then $|C^{^{\wedge}}_{G}(x^g) |=| C^{^{\wedge}}_{G}(x) | $ and $ D^{^{\wedge}}_{n}(C^{^{\wedge}}_{G}(x))=D^{^{\wedge}}_{n}(C^{^{\wedge}}_{G}(x^g)) $ for all $n\geq0$, and hence the result holds by Lemma \ref{main0}.
\end{proof}
The following technical result is a generalization of Theorem \ref{prim3} and play important rule in proving the main result.
\begin{theorem}\label{main2}
 Assume that $G$ is a finite group and $p$ is the smallest prime number dividing the order of $G$. Then $$D^{^{\wedge}}_{n}(G)\leq\frac{1}{p^{n-1}}d(G)+\frac{(1-p)|Z(G)|-(1-p^nD^{^{\wedge}}_{n-1}(G))|Z^{^{\wedge}}(G)|}{p^n|G|}.$$
\end{theorem}
\begin{proof}
Since for all  $x_i\in G\setminus Z^{^{\wedge}}(G)$ one has $[G:C^{^{\wedge}}_G(x_i)]\geq p$, by Lemma \ref{main1} we have,
\begin{eqnarray*}
D^{^{\wedge}}_n(G)&=&\frac{1}{|G|^n} \sum_{i=1}^{k(G)}\frac{|C^{^{\wedge}}_{G}(x_i)|^n}{|C_{G} (x_i) |} ~D^{^{\wedge}}_{n-1}(C^{^{\wedge}}_{G}(x_i))\vspace{.3cm}\\
&\leq&\frac{|Z^{^{\wedge}}(G)|}{|G|}D^{^{\wedge}}_{n-1}(G)+\frac{|Z(G)|-|Z^{^{\wedge}}(G)|}{|G|p^n}+\frac{k(G)-|Z(G)|}{|G|p^{n-1}}\vspace{.3cm}\\
&=&\frac{1}{p^{n-1}}d(G)+\frac{|Z^{^{\wedge}}(G)|}{|G|}D^{^{\wedge}}_{n-1}(G)+\frac{(1-p)|Z(G)|-|Z^{^{\wedge}}(G)|}{p^n|G|}\vspace{.3cm}\\
&=&\frac{1}{p^{n-1}}d(G)+\frac{(1-p)|Z(G)|-(1-p^nD^{^{\wedge}}_{n-1}(G))|Z^{^{\wedge}}(G)|}{p^n|G|}.
 \end{eqnarray*}
\end{proof}
\begin{theorem}
Assume that $G$ is a non-cyclic finite group and $p$ is the smallest prime number dividing the order of $G$. Then $$D^{^{\wedge}}_n(G)\leq\frac{p^{n+1}+p^n-1}{p^{2n+1}}$$ and the equality holds if and only if $G/Z^{^{\wedge}}(G)$ is an elementary abelian $p$-group of rank $2$.
\end{theorem}
\begin{proof}
We proceed by induction on $n$. For $n=1$ the result follows by Corollary \ref{coro}. Now assume that the result holds for $n-1$. First suppose that $G$ is abelian, by Theorem \ref{main2} we have
$$D^{^{\wedge}}_n(G)\leq\frac{1}{p^{n-1}}+\frac{(1-p)}{p^n}+\left(p^n\left(\frac{p^n+p^{n-1}-1}{p^{2n-1}}\right)-1\right)\frac{|Z^{\wedge}(G)|}{p^n|G|}.$$
By using Beyl and Tappe (\cite{beta}, proposition 4.9c), $|Z^{^{\wedge}}(G)|/|G|\leq1/p^2$ and so
$$D^{^{\wedge}}_n(G)\leq\frac{1}{p^{n}}+\frac1{p^2}\left(\frac{p^n+p^{n-1}-1}{p^{2n-1}}-\frac{1}{p^{n}}\right)=\frac{p^{n+1}+p^n-1}{p^{2n+1}}.$$
Now if $G$ is non-abelian, then since $D^{^{\wedge}}_n(G)\leq d_n(G)$, the result holds by Proposition \ref{prim2}.

Now assume that $G/Z^{^{\wedge}}(G)$ is an elementary abelian $p$-group of rank 2. Then for all $x\in G\setminus Z^{^{\wedge}}(G)$, $[G:C^{^{\wedge}}_G(x)]=p$ and $C^{^{\wedge}}_G(x)/Z^{^{\wedge}}(G)=\langle xZ^{\wedge}(G)\rangle$ is a cyclic group of order $p$ . For $a, b\in C^{^{\wedge}}_G(x)$ one has $a=x^iz_1$ and $b=x^jz_2$ for some $z_1,z_2\in Z^{^{\wedge}}(G)$ and $0\leq i,j\leq p-1$. Therefore $a\wedge b=x^iz_1\wedge x^jz_2=x^i\wedge x^j=1$ and so $D^{^{\wedge}}_{n-1}(C^{^{\wedge}}_G(x))=1$ for all $x\in G\setminus Z^{^{\wedge}}(G)$ and $n\geq1$. Hence all inequalities in the above proof and the proof of Theorem \ref{main2} turn in to equality.

Conversely, assume that $D^{^{\wedge}}_n(G)=({p^{n+1}+p^n-1})/{p^{2n+1}}$.Then
\begin{align*}
\frac{p^{n+1}+p^n-1}{p^{2n+1}}=D^{^{\wedge}}_n(G)&= \frac{1}{|G|^{n+1}} \sum_{x\in G}|C^{{^\wedge}}_{G}(x)|^n D^{^{\wedge}}_{n-1}(C^{^{\wedge}}_{G}(x))\\
&\leq \frac{1}{p^n}+\frac{|Z^{^{\wedge}}(G)|}{|G|}\left(D^{^{\wedge}}_{n-1}(G)-\frac{1}{p^n}\right)
\end{align*}
Since $G$ is not cyclic, $D^{^{\wedge}}_{n-1}(G)\leq(p^n+p^{n-1}-1)/p^{2n-1}$ and so
$$\frac{p^{n+1}+p^n-1}{p^{2n+1}}\leq\frac{1}{p^n}+\frac{|Z^{^{\wedge}}(G)|}{|G|}\left(\frac{p^n-1}{p^{2n-1}}\right).$$
It follows that $[G:Z^{^{\wedge}}(G)]\leq p^2$. On the other hand by \cite{beta}, Proposition 4.9c we have $[G:Z^{^{\wedge}}(G)]\geq p^2$. Therefore  $G/Z^{^{\wedge}}(G)$ is an elementary abelian $p$-group of rank 2.
\end{proof}
\begin{theorem}
Assume that $p$ is the smallest prime number dividing the order of non-abelian finite group $G$ and $Z^{^{\wedge}}(G)$ is proper subgroup of $Z(G)$. Then
 $$D^{^{\wedge}}_n(G)\leq\frac{p^{2n+1}(p+1)+p^{2n}-1}{p^{3n+1}(p+1)}.$$
\end{theorem}
\begin{proof}
Use induction on $n$. Since $G$ is not unicentral and abelian, $d(G)\leq (p^2+p-1)/p^3$ and $[Z(G):Z^{^{\wedge}}(G)]\geq p$. For $n=1$ the result follows by Corollary \ref{coro}. Now by Theorem \ref{main2} we have
$$D^{^{\wedge}}_n(G)\leq\frac{1}{p^{n-1}}\left(\frac{p^2+p-1}{p^3}\right)+\frac{|Z(G)|}{|G|}\left(\frac{p-p^2-1+p^n D^{^{\wedge}}_{n-1}(G)}{p^{n+1}}\right).$$
Since $G$ is non-abelian, $[G:Z(G)]\geq p^2$ and by the induction hypothesis,
\begin{eqnarray*}
D^{^{\wedge}}_n(G)&\leq&\frac{p^2+p-1}{p^{n+2}}+\frac1{p^2}\left(\frac{p^{2n-1}(2p+1)-p^{2n}(p+1)-1}{p^{3n-1}(p+1)}\right)\vspace{.3cm}\\
&=&\frac{p^{2n+1}(p+1)+p^{2n}-1}{p^{3n+1}(p+1)}.
\end{eqnarray*}
\end{proof}

When $G$ is a capable group, Our result gives a sharper upper bound for multiple exterior degree as the following.

\begin{theorem}
Let $G$ be a non-abelian capable group and $p$ be the smallest prime number dividing the order of $G$. Then $D^{^{\wedge}}_n(G)\leq 1/p^n.$
\end{theorem}

\begin{proof}
Use induction on $n$. By Proposition \ref{prop1} and Theorem \ref{theo} one may assume that $Z(G)\cap G^{'}\neq1$. Assume that $\mathcal{C}=\{x_1,...,x_{k(G)}\}$ is a system of representatives for the conjugacy classes of a finite group $G$. Then thanks to Theorem 2.8 in \cite{pr} there is $x_i\in \mathcal{C}-Z(G)$ such that $[C_G(x_i):C^{^{\wedge}}_G(x_i)]\geq p$. Therefore
\begin{align*}
D^{^{\wedge}}_n(G)&= \frac{1}{|G|^n} \sum_{i=1}^{k(G)}\frac{|C^{^{\wedge}}_{G}(x_i)|^n}{|C_{G} (x_i) |} ~D^{^{\wedge}}_{n-1}(C^{^{\wedge}}_{G}(x_i))\\
&\leq \frac{1}{|G|^n}\left(|Z^{^{\wedge}}(G)||G|^{n-1}D^{^{\wedge}}_{n-1}(G)+\frac{|G|^{n-1}}{p^{n}}(|Z(G)|-|Z^{^{\wedge}}(G)|)+
\frac{|G|^{n-1}}{p^{2n-1}}\right)\\
&+\frac{1}{|G|^n}\left(\frac{|G|^{n-1}}{p^{n-1}}(k(G)-|Z(G)|-1)\right).
\end{align*}
Since $G$ is capable, $Z^{^{\wedge}}(G)=1$ and so by inductive hypothesis,
\begin{align*}
D^{^{\wedge}}_n(G)&\leq \frac{1}{|G|}\left(\frac{1}{p^{n-1}}+\frac{1}{p^{n}}(|Z(G)|-1|)+
\frac{1}{p^{2n-1}}+\frac{1}{p^{n-1}}(k(G)-|Z(G)|-1)\right)\\
&\leq \frac{1}{p^{n-1}}d(G)+\left(\frac{1-p}{p^n}\right)\frac{|Z(G)|}{|G|}.
\end{align*}
On the other hand $$d(G)\leq \frac{1}{p}+\left(\frac{p-1}{p}\right)\frac{|Z(G)|}{|G|},$$ and so the result follows.
\end{proof}

\begin{theorem}
If the orders of the groups $G$ and $H$ are coprime, then $$D^{^{\wedge}}_n(G\times H)= D^{^{\wedge}}_n(G)D^{^{\wedge}}_n(H). $$
\end{theorem}
\begin{proof}
We proceed by induction on $n$. For $n=1$, it has been proved in \cite{pr} Lemma 2.10. Since  $C^{^{\wedge}}_{G\times H}(g,h)= C^{^{\wedge}}_{G}(g) \times C^{^{\wedge}}_{H}(h)$, by Theorem \ref{main0} and inductive hypothesis we have
\begin{align*}
D^{^{\wedge}}_n(G\times H)&=\frac{1}{|G\times H|^{n+1}} \sum_{(g,h)\in G\times H } |C^{^{\wedge}}_{G\times H}(g,h)|^n D^{^{\wedge}}_{n-1}(C^{^{\wedge}}_{G\times H}(g,h))\\
&=\frac{1}{|G|^{n+1}|H|^{n+1}} \sum_{g\in G} \sum_{h\in H}|C^{^{\wedge}}_{G}(g) |^n |C^{^{\wedge}}_{H}(h) |^n D^{^{\wedge}}_{n-1}(C^{^{\wedge}}_{G}(g)) ~D^{^{\wedge}}_{n-1}(C^{^{\wedge}}_{H}(h))\\&=
\left(\frac{1}{|G|^{n+1}} \sum_{g\in G}|C^{^{\wedge}}_{G}(g) |^n D^{^{\wedge}}_{n-1}(C^{^{\wedge}}_{G}(g))\right)\\ 
&\times\left(\frac{1}{|H|^{n+1}}\sum_{h\in H}|C^{^{\wedge}}_{H}(h) |^n D^{^{\wedge}}_{n-1}(C^{^{\wedge}}_{H}(h))\right)\\
&= D^{^{\wedge}}_n(G)D^{\wedge}_n(H).
\end{align*}
\end{proof}
\section{Some examples }
\begin{example}
\rm{Let $D_{2n}=\langle a,b | a^2=b^n= (ab)^2=1 \rangle$ be the dihedral group of order $2n.$ By Example 3.1 in \cite{pr} we have $Z^{\wedge}(D_{2n})=1$ and $d(D_{2n})=d^{\wedge}(D_{2n})$. Furthermore, $C^{^{\wedge}}_{D_{2n}}(b^i)=\langle b\rangle$ and $C^{^{\wedge}}_{D_{2n}}(ab^i)=\langle ab^i\rangle$. By using induction on $m$ we prove that $$D^{^{\wedge}}_m(D_2n)=\frac{n^m+2^{m+1}-1}{2(2n)^m}.$$
\begin{align*}
D^{^{\wedge}}_m(D_2n)&= \frac{1}{|D_{2n}|^{m+1}} \sum_{x\in D_{2n}} |C^{^{\wedge}}_{D_{2n}}(x)|^m D^{^{\wedge}}_{m-1}(C^{^{\wedge}}_{D_{2n}}(x))\\
&=\frac{1}{(2n)^{m+1}}\left((2n)^mD^{^{\wedge}}_{m-1}(D_{2n})+(n-1)n^m+2^m n \right)\\
&=\frac{1}{(2n)^{m+1}}\left((2n)^m\left(\frac{n^{m-1}+2^m-1}{2(2n)^{m-1}}\right)+(n-1)n^m+2^m n \right)\\
&=\frac{n^m+2^{m+1}-1}{2(2n)^m}.
\end{align*}}
\end{example}
\begin{example}
\rm{Let $Q_n=\langle a,b | a^n=b^2=(ab)^2\rangle$ be the generalized quaternion group of order $4n$, $n\geq1$ that its Schur multiplier is trivial and so is multiple unidegree group. By the same computation as in the above example, we have
$$D^{^{\wedge}}_m(Q_n)=d_m(Q_n)=\frac{n^m+2^{m+1}-1}{2(2n)^m}$$  }
\end{example}


\end{document}